\documentclass{amsart}
\usepackage[all]{xy}

\newcounter{zlist}

\newcounter{blist}

\newcounter{rlist}

\swapnumbers
\newtheorem{theorem}{Theorem}[section]
\newtheorem{lemma}[theorem]{Lemma}

\newtheorem{proposition}[theorem]{Proposition}

\newtheorem{remark}[theorem]{Remark}
\numberwithin{equation}{section}

\newcommand{\A}{{\mathcal{A}}}
\newcommand{\B}{{\mathcal{B}}}
\newcommand{\M}{{\mathfrak{M}}}
\newcommand{\X}{{\mathfrak{X}}}
\begin{document}

\title{Pure morphisms are effective for modules}
 \author{Bachuki Mesablishvili}
\thanks{The work was partially supported by Volkswagen Foundation (Ref.: I/85989) and  Shota Rustaveli National Science Foundation Grant DI/12/5-103/11}

\begin{abstract} Yet another proof of the result asserting that a morphism of commutative rings is an
effective descent morphism for modules if and only if it is pure is given. Moreover, it is shown that this
result cannot be derived from Moerdijk's descent criterion.
\end{abstract}

\smallskip


\maketitle

\section{Introduction}

Let $\A$ be a category with pullbacks. An $\A$-indexed category $\X$ is a pseudo-functor $\A^{\text{op}} \to \mathbf{CAT}$,
where $\mathbf{CAT}$ denotes the 2-category of locally small (but possibly large) categories,
explicitly given by the data of a family of categories $\X^A$,
indexed by the objects of $\A$, with \emph{change of base functors} $p^*:\X^B \to \X^E$ for each
morphism $p:E \to B$ of $\A$ and with additional structure expressing the idea of a
pseudo-functor (see, \cite{McP}, \cite{PS}). For example, $\A^{\text{op}} \to \mathbf{CAT}$ given by $A \longmapsto \A/A$ and $(p:E \to B) \longmapsto (p^* : \A/B \to \A/E)$,
the pullback functor along $p$, is an $\A$-indexed category, also called the \emph{basic} $\A$-\emph{indexed category}.

If $\mathbf{C}$ is an internal category of $\A$ (e.g., \cite[Chapter XII]{Mc}), then one defines an ordinary category $\A^\mathbf{C}$ of $\mathbf{C}$-diagrams
in $\A$, and the assignment $\mathbf{C} \to \A^\mathbf{C}$ induces a pseudo-functor
 $$\X^{(-)}:\mathbf{cat}(\A)^{\text{op}}\to \mathbf{CAT}$$
of 2--categories (see \cite{JT2}), where $\mathbf{cat}(\A)$ denotes the 2-category of internal categories of $\A$.
Let $p : E \to B$ be a morphism in $\A$. Then, $p$ gives rise to an internal category $\mathbf{Eq}(p)$
of $\A$, namely the equivalence relation induced by $p$, and to a fully faithful (internal) functor
$p :\mathbf{Eq}(p) \to \mathbf{B}$, where $\mathbf{B}$ is the discrete internal category on $B$. The category $\X^{\mathbf{Eq}(p)}$ is
called the \emph{category of} $\X$-\emph{descent data relative to} $p$, and denoted by $\mathrm{Des}_\A(p)$.
The pseudo-functor $\X^{(-)}:\mathbf{cat}(\A)^{\text{op}}\to \mathbf{CAT}$
carries $p$ into an ordinary functor $$K_p : \X^B \to \mathrm{Des}_\A(p),$$
and one says that $p$ is an \emph{effective} $\X$-\emph{descent morphism} if $K_p$ is an equivalence of categories. When $\X$
is the base indexed category, one also says that $p$ is an \emph{effective descent morphism in} $\A$.

In the case of an $\A$-indexed category $\X$ satisfying the Beck-Chevalley condition (which requires that for each
$p:E \to B$ in $\A$, the functor $p^*:\X^B \to \X^E$ admits a left adjoint $p_!: \X^E \to \X^B$, and
for every pullback diagram
$$\xymatrix{E' \ar[r]^{q} \ar[d]_{g} & B' \ar[d]^f \\
E\ar[r]_p &B,}$$  the canonical map $q_!g^* \to f^*p_!$ is an isomorphism, and which is always satisfied by the basic $\A$-indexed category), then there exists an
equivalence of categories between the category  $\mathrm{Des}_\A(p)$ of descent data for $p$ in $\A$,
and the category $(\X^E)^{\textbf{T}_p}$ of $\textbf{T}_p$-algebras for the monad on $\X^E$ induced by the adjoint pair
$p_! \dashv p^*:\X^E \to \X^B$. Moreover, modulo this equivalence, the functor $K_p : \X^B \to \mathrm{Des}_\A(p)$
corresponds to the comparison functor $\X^B \to (\X^E)^{\textbf{T}_p}$. Thus, $p$ is an effective $\X$-descent
morphism if and only if the functor $p^*:\X^B \to \X^E$ is monadic. Hence the descent problem
of determining whether a given morphism is effective for descent
 w.r.t. an indexed category satisfying the Beck-Chevalley condition can be reduced to
the problem of monadicity of a suitable functor. Consequently, in verifying that a particular morphism is
effective for descent, the various versions of Beck's Monadicity Theorem are extremely useful. Note that the crucial requirement of
each of these versions is the preservation of some kind of coequalizers.

Now let $\mathcal{E}$ be the category dual to the category of commutative rings (with 1) and consider the $\mathcal{E}$-indexed category
$\M:\mathcal{E}^{\text{op}} \to \mathbf{CAT}$ that assigns to each commutative ring $B$ the category of $B$-modules, $\textbf{Mod}_B$,
and to any morphism $p:B \to E$ of commutative rings the \emph{extension-of-scalars functor} $E\otimes_B -:\textbf{Mod}_B \to \textbf{Mod}_E$.
In this case, one sometimes speaks of \emph{effective descent for modules}. It
is well known that $\M$ satisfies the dual of the Beck-Chevalley condition. Thus, $p$ is an effective descent morphism for modules iff the functor
$E\otimes_B -$ is comonadic.

It is proved in \cite{M1} that effective descent morphisms for modules are precisely the pure monomorphisms. The
proof makes heavy use of the description of purity by means of the functor $\mathrm{Hom}_\mathbb{Z}(-, \mathbb{Q/\mathbb{Z}}):\textbf{Ab} \to \textbf{Ab}^{\text{op}}$ to check the needed preservation of certain coequalizer diagrams.

In this paper, we give a different proof of this result which uses another characterization of pure morphisms, namely that
pure morphisms of modules are precisely filtered colimits of split monomorphisms. The key ingredient in our proof is the observation
that for the comonadicity of an additive faithful functor between abelian categories, one only needs preservation of certain monomorphisms, not some kind of equalizer diagrams. This follows from a variation of Duskin's theorem (see Theorem \ref{Ab}). By analyzing these particular monomorphisms, we derive
that they all are pure, thus preserved by the extension-of-scalars functors. This result together with the fact that for any morphism of commutative rings, the corresponding extension-of-scalars functor is faithful if and only if the morphism is pure, imply the comonadicity of the extension-of-scalars functors.

In \cite{Mk}, I. Moerdijk gave four conditions which, if satisfied by some class $\mathfrak{C}$ of morphisms in a category $\A$ with pullbacks,
imply that the members of $\mathfrak{C}$ are effective descent morphisms in $\A$. Since a morphism of commutative rings
is an effective descent morphism in $\mathcal{E}$ if and only if it is an effective descent morphism for modules (see \cite{JT3}),
it is natural to ask whether these conditions are satisfied by the class of pure monomorphisms of commutative rings. We show in Section 3
that the answer is negative.

As background to the subject, we refer to S. MacLane \cite{Mc} for
generalities on category theory, to \cite{McP}, \cite{PS} for terminology and general
results on indexed categories and to G. Janelidze and W. Tholen \cite{JT1}, \cite{JT2} and \cite{JT3} for descent theory.

\section{Main result}

Let $p:B \to E$ be a morphism of commutative rings. Write $p^*$ for the \emph{extension-of-scalars functor} $$E \otimes_B - :\textbf{Mod}_B \to \textbf{Mod}_E.$$
Recall that a morphism $f : M \to M'$ of $B$-modules is \emph{pure} if for any $B$-module $N$, $$N\otimes_B f: N\otimes_B M \to N \otimes_B M'$$ is monic.

\begin{lemma}\label{L} If $p:B \to E$ is a pure morphism of commutative rings (i.e. $p$ is pure as a morphism
of $B$-modules), then the equalizer of any $p^*$-split pair of morphisms of $B$-modules is pure.
\end{lemma}

\begin{proof} If  $p:B \to E$ be a pure morphism of commutative rings, then by \cite[Proposition 2.30 (ii) ]{AR} there is a directed diagram of morphisms $(p_d:B \to E_d)_{d \in \mathcal{D}}$
of $B$-modules with connecting morphisms
$(1_B,e_{d,d'}):p_d \to p_{d'}$ for $d\leq d'$ such that each $p_d$ is a split monomorphism (say, with splitting $q_d$) and
$p$ is a colimit of this diagram, say, with colimit morphisms

$$\xymatrix {B \ar[r]^{p_d}\ar@{=}[d]& E_d \ar[d]^{\kappa_d} \\
B \ar[r]_p &E, & d \in \mathcal{D}.}$$

Suppose now that
$$
\xymatrix { X    \ar[r]^{f}&    Y  \ar@{->}@<0.5ex>[r]^{g}\ar@<-0.7ex>[r]_{h}& Z }
$$ is an equalizer diagram  of $B$-modules such that there is a split equalizer diagram of $E$-modules

$$\xymatrix {M \ar[r]^-{\overline{f}}&  E\otimes_B Y  \ar@{->}@<0.5ex>[rr]^{E\otimes_B g}\ar@<-0.7ex>[rr]_{E\otimes_B h}&& E\otimes_B Z .}
$$  Then, in $\textbf{Mod}_B$, we have the following commutative diagram

\begin{equation}\label{eq.1}
\xymatrix {X \ar[d]_{s} \ar[rr]^-{f}& &      Y \ar[d]|{p\otimes_B Y} \ar@{->}@<0.5ex>[rr]^{ g}\ar@<-0.7ex>[rr]_{h}&&  Z \ar[d]|{p\otimes_B Z}\\
M  \ar[rr]_-{\overline{f}}& &      E \otimes_B Y  \ar@{->}@<0.5ex>[rr]^{E\otimes_B g}\ar@<-0.7ex>[rr]_{E\otimes_B h}&& E\otimes_B Z}
\end{equation} for some $s:X \to M.$

For any $d \in \mathcal{D}$, let

$$\xymatrix {M_d \ar[r]^-{f_d} &E_d \otimes_B Y  \ar@{->}@<0.5ex>[r]^{E_d \otimes_B g}\ar@<-0.7ex>[r]_{E_d \otimes_B h}& E_d \otimes_B Z }$$ be
an equalizer diagram. Then it is not hard to see that  $$(f_d, (v_{d,d'},e_{d,d'}\otimes_B Y):f_d \to f_{d'})_{d\leq d'\in \mathcal{D}}$$ is a directed diagram of morphisms of $B$-modules, where
connecting morphisms $v_{d,d'}:M_d \to M_{d'}$ ($d\leq d'$) are the comparison morphisms induced by the universal property of equalizers:

$$\xymatrix {M_d \ar@{..>}[d]|{v_{d,d'}} \ar[rr]^-{f_d}& &      E_d\otimes_B Y \ar[d]|{e_{d,d'}\otimes_B Y} \ar@{->}@<0.5ex>[rr]^{E_d\otimes_B g}\ar@<-0.7ex>[rr]_{E_d\otimes_B h}&& E_d\otimes_B Z \ar[d]|{e_{d,d'}\otimes_B Z}\\
M_{d'}  \ar[rr]_-{f_{d'}}& &      E_{d'} \otimes_B Y  \ar@{->}@<0.5ex>[rr]^{E_{d'}\otimes_B g}\ar@<-0.7ex>[rr]_{E_{d'}\otimes_B h}&& E_{d'}\otimes_B Z.}
$$

\noindent For each $d \in \mathcal{D}$, write $\iota_d:M_d \to M$ for the comparison morphism  making the diagram

$$\xymatrix {M_d \ar@{..>}[d]|{\iota_{d}} \ar[rr]^-{f_d}& &      E_d\otimes_B Y \ar[d]|{\kappa_{d}\otimes_B Y} \ar@{->}@<0.5ex>[rr]^{E_d\otimes_B g}\ar@<-0.7ex>[rr]_{E_d\otimes_B h}&& E_d\otimes_B Z \ar[d]|{\kappa_{d}\otimes_B Z}\\
M  \ar[rr]_-{\overline{f}}& &      E \otimes_B Y  \ar@{->}@<0.5ex>[rr]^{E\otimes_B g}\ar@<-0.7ex>[rr]_{E\otimes_B h}&& E\otimes_B Z}
$$ commute. Since $E$ is a directed colimit of the $E_d$'s  and since directed colimits commute with equalizers (e.g. \cite{AR}), taking directed
colimit in the last diagram gives

$$\overline{f}={\varinjlim}_{d\in \mathcal{D}}(f_d, (v_{d,d'},e_{d,d'}\otimes_B Y):f_d \to f_{d'}).$$

\noindent Consider now the following commutative diagram

$$\xymatrix { X \ar@<-.7ex>[dd]_{s_d}    \ar[rr]^{f}&&    Y \ar@<-.7ex>[dd]_{p_d\otimes_B Y} \ar@{->}@<0.5ex>[rr]^{g}\ar@<-0.7ex>[rr]_{h}&& Z \ar@<-.7ex>[dd]_{p_d\otimes_B Z}\\\\
M_d   \ar@<-.7ex>[uu]_{t_d}         \ar[rr]_{f_d}& &      E_d\otimes_B Y \ar@<-.7ex>[uu]_{q_d\otimes_B Y} \ar@{->}@<0.5ex>[rr]^{E_d\otimes_B g}\ar@<-0.7ex>[rr]_{E_d\otimes_B h}&& E_d\otimes_B Z \ar@<-.7ex>[uu]_{q_d\otimes_B Z},}
$$ in which $s_d$ and $t_d$ are the comparison morphisms induced by the universal property of equalizers.
This gives rise to a directed  diagram $$(s_d, (1_X, v_{d,d'}):s_d \to s_{d'})_{d\leq d'\in \mathcal{D}}.$$ Since $\varinjlim_{d\in \mathcal{D}}f_d=\overline{f}$ and since directed colimits commute with equalizers, taking the directed colimit in the last diagram, we obtain $$s=\varinjlim_{d\in \mathcal{D}}(s_d, (1_X, v_{d,d'}):s_d \to s_{d'}).$$ Since $q_dp_d=1$, $t_ds_d=1$, and hence each $s_d$ is a split monomorphism, implying that $s$ is pure. Looking now at the left square in Diagram (\ref{eq.1}) and using that
\begin{itemize}
  \item the morphism $\overline{f}$, being a split monomorphism, is pure;
  \item pure morphisms are closed under composition (see \cite{AR});
  \item pure morphisms are left cancellative (see \cite{AR}),
\end{itemize} one obtains that $f$ is also pure.

\end{proof}

We need the the following variation of Duskin's
theorem (see \cite{M}).

\begin{theorem} \label{Ex}Let $\A$ be a category admitting kernel-pairs
of split epimorphisms and let $\B$ be an exact category. Then the
following two assertions are equivalent for any right adjoint
functor $U: \B \to \A$:

\begin{itemize}
\item [(i)] $U$ is monadic.

\item [(ii)] $U$ is conservative and $U$ preserves those
regular epimorphisms whose kernel-pairs are $U$-split.
\end{itemize}
\end{theorem}

Since any abelian category is coexact and since in such a category regular monomorphisms coincide with monomorphisms,
it follows from the dual of Theorem \ref{Ex} that

\begin{theorem}\label{Ab} A left adjoint additive functor $F: \A \to \B$ between abelian categories is comonadic if and only if $F$ is conservative
and $F$ preserves those monomorphisms whose cokernel-pairs are $F$-split.
\end{theorem}

Let $\mathcal{E}$ be the category dual to the category of commutative rings (with 1). It is well-known that the category
of commutative rings has pushouts: If $B \to E$ and $B \to E'$
are morphisms of commutative rings, then their pushout is the tensor product $E \otimes_B E'$. Thus the category
$\mathcal{E}$ admits pullbacks and the assignments $B \longmapsto \textbf{Mod}_B$ and $(p:B \to E) \longmapsto (p^*=E\otimes_B- :\textbf{Mod}_B
 \to \textbf{Mod}_E)$ define an $\mathcal{E}$-indexed category $\mathcal{E}^{\text{op}} \to \mathbf{CAT}$, which we shall denote by $\M$.
We say that a morphism of commutative rings is an \emph{effective descent morphism for modules} if it is an effective $\M$-descent morphism. As is well known,  $\M$ satisfies the dual of the
Beck-Chevalley condition. Thus, a morphism $B \to E$ is an effective descent morphism for modules iff
the extension-of-scalars functor $p^* :\textbf{Mod}_B \to \textbf{Mod}_E$ is comonadic.

\,
\,

We are now ready to state our main result:
\begin{theorem} \label{main}A morphism of commutative rings is an effective descent morphism for modules if and only if it is pure.
\end{theorem}
\begin{proof} We note first that a morphism of commutative rings is pure if and only if
the corresponding extension-of-scalars functor is conservative (e.g., \cite{JT3}). Thus one direction is trivial.

For the converse, let $p:B \to E$ be a pure morphism of commutative rings. Then the functor
$p^*=E \otimes_B - :\textbf{Mod}_B \to \textbf{Mod}_E$ is conservative. Thus, in order to be able to apply
Theorem \ref{Ab} to prove the comonadicity of $p^*$, we have only to show that $p^*$ preserves those monomorphisms whose cokernel-pairs are $p^*$-split.
So suppose that $f:X \to Y$ is a monomorphism of $B$-modules whose cokernel-pair $
\xymatrix {Y  \ar@{->}@<0.5ex>[r]^-{i_1}\ar@<-0.7ex>[r]_-{i_2}& Y\sqcup_XY }$ is $p^*$-split. Then since the diagram
$$
\xymatrix { X    \ar[r]^{f}&    Y   \ar@{->}@<0.5ex>[r]^-{i_1}\ar@<-0.7ex>[r]_-{i_2}& Y\sqcup_XY  }
$$ is an equalizer diagram, we conclude by Lemma \ref{L} that $f$ is a pure morphism of $B$-modules.
Then it is a directed colimit of split monomorphisms of $B$-modules and
hence $p^*(f)=E \otimes_B f$ is a directed colimit of split monomorphisms of $E$-modules, thus a monomorphism in $\textbf{Mod}_E$.
Consequently, the functor $p^*$ is comonadic. So $p$ is an effective descent morphism for modules, as required.
\end{proof}

\section{Moerdijk's conditions}

Let $\A$ be a category with pullbacks. We say that a morphism $p:E \to B $ in $\A$ is an
\emph{effective descent morphism in} $\A$ if $p$ is an effective descent morphism with respect to the basic $\A$-indexed category,
which is $-$since the Beck-Chevalley condition is always satisfied
by the basic $\A$-indexed category$-$ equivalent to saying that
the change-of-base functor $$p^*=E\times_B -:\A/B \to \A/E$$
given by pulling back along $p$, is comonadic.

In \cite{Mk}, I. Moerdijk proved that if a class $\mathfrak{C}$ of morphisms in $\mathcal{A}$ satisfies the following conditions:
\begin{itemize}
  \item [](I) $\mathfrak{C}$ contains
the isomorphisms and is closed under composition,
  \item [](II) every morphism in $\mathfrak{C}$ is a regular epimorphism in $\mathcal{A}$,
  \item [](III) $\mathfrak{C}$ is stable under pullback, and
  \item [](IV) the coequalizer of a parallel pair of $\mathfrak{C}$-morphisms exists in $\mathcal{A}$, and is stable
under pullback along $\mathfrak{C}$-morphisms,
\end{itemize} then the members of $\mathfrak{C}$ are effective descent morphisms in $\A$.

\,
\,

\noindent We recall the following result:

\begin{theorem} \label{JT}\emph{(}\cite{JT3}\emph{)} The following are equivalent for a morphism $p: B \to E$ of commutative rings:
\begin{itemize}
  \item [(i)] $p: B \to E$ is a pure morphism of $B$-modules,
  \item  [(ii)] $p: B \to E$ is effective descent morphism for modules,
  \item [(iii)] $p: B \to E$ is effective descent morphism in the dual of the category of commutative rings, $\mathcal{E}$.
\end{itemize}
\end{theorem}

In view of Theorem \ref{JT} it is natural to ask if the class $\mathfrak{P}$ of $\mathcal{E}$-morphisms determined by
pure morphisms of commutative rings satisfies Conditions I-IV. The answer is no, and now we wish to
prove it.

\begin{lemma}\label{L1}Let $\mathcal{A}$ be a category with finite sums and pullbacks that are distributive over finite sums and
 $\mathfrak{C}$ be a class of morphisms in $\mathcal{A}$ satisfying Condition \emph{(IV)}. If
$\mathfrak{C}$ contains split epimorphisms, then for any morphism $p:E \to B$ in $\mathfrak{C}$, the functor
$$E \times_B -: \mathcal{A}/B \to \mathcal{A}/B$$ preserves all coequalizers.
\end{lemma}

\begin{proof}Note first that since in slice categories, coequalizers are computed as in $\mathcal{A}$, it follows
in particular that being a coequalizer in a slice category
is equivalent to being a coequalizer in $\mathcal{A}$.

Next, since pullbacks distribute over finite sums in $\mathcal{A}$, the dual of the argument in the proof of \cite[Theorem 4.1]{B}
shows that the functor $E \times_B -: \mathcal{A}/B \to \mathcal{A}/B$ preserves all coequalizers if and only if it preserves coequalizers
of reflexive pairs. But if $\xymatrix {X  \ar@{->}@<0.5ex>[r]^{f}\ar@<-0.7ex>[r]_{g}& Y\ar[r]^{h}& Z }$ is such a coequalizer in $\mathcal{A}/B$,
then since both $f$ and $g$ are split epimorphisms, they belong to $\mathfrak{C}$, whence one concludes by Condition (IV) that
$$\xymatrix {E \times_B X\ar@{->}@<0.5ex>[r]^{E \times_B f}\ar@<-0.7ex>[r]_{E \times_B g}& E \times_B Y \ar[r]^{E \times_B h}& E \times_B Z  }$$
is a coequalizer diagram (in $\mathcal{A}$, and hence in $\mathcal{A}/B$). Thus the functor
$E \times_B -: \mathcal{A}/B \to \mathcal{A}/B$ preserves all coequalizers
\end{proof}

\begin{proposition} \label{Pure}The class $\mathfrak{P}$ does not satisfies Condition \emph{(IV)}.
\end{proposition}
\begin{proof} We shall assume that the class $\mathfrak{P}$ satisfies Condition (IV), and use this assumption to derive a
contradiction.

Recall first that the category of commutative rings admits  also finite products and that pushouts distribute
over these products. Thus the category $\mathcal{E}$ admits pullbacks and finite sums, and pullbacks distribute
over finite sums.

Now, if Condition (IV) is satisfied by the class $\mathfrak{P}$, then it follows from Lemma \ref{L1} that
for any $\mathfrak{P}$-morphism $p: E \to B$, the functor $E \times_B -: \mathcal{E}/B \to \mathcal{E}/B$ preserves all coequalizers, which
is $-$since $(\mathcal{E}/B)^{\text{op}}$ is just  the category $B\textbf{-\text{Alg}}$ of commutative $B$-algebras$-$ equivalent to saying that
the functor $E \otimes_B -: B\textbf{-\text{Alg}} \to B\textbf{-\text{Alg}}$
preserves all equalizers, whence one concludes by \cite[Proposition 4.6]{CS} that
$E$ is a (faithfully) flat $B$-module. Thus, under Condition (IV), any pure morphism of commutative rings is faithfully flat.
But this is not true (see, for example, Exercise 4.8.13 in \cite{Bo}). This contradiction
completes the proof of the proposition.
\end{proof}

\begin{remark} {\em A much easier proof of Proposition \ref{Pure} has been provided by G. Janelidze. His idea is to
consider the functor $(-)_B :\textbf{Mod}_B \to B\textbf{-\text{Alg}}$ that takes a $B$-module $M$
to the semidirect product $M_B=B\oplus M$ of the ring $B$ with $M$ (see, \cite[p. 32]{CS}). Suppose now that
$p: E \to B$ is a $\mathfrak{P}$-morphism for which the functor $E \otimes_B -: B\textbf{-\text{Alg}} \to B\textbf{-\text{Alg}}$
preserves all equalizers and consider an arbitrary monomorphism $f: M \to M'$ of $B$-modules. Since monomorphisms in
$\textbf{Mod}_B$ are equalizers, and since the functor  $(-)_B$ preserves and reflects equalizer diagrams (\cite{CS}),
$f_B : M_B \to M'_B$ is an equalizer in $B\textbf{-\text{Alg}}$. Then $E \otimes_B f_B : E \otimes_B M_B \to E \otimes_B M'_B$
is also an equalizer (in $B\textbf{-\text{Alg}}$ and hence in $\textbf{Mod}_B$). Considering now the commutative diagram
$$\xymatrix{E \otimes_B M \ar[rr]^-{E \otimes_B f}\ar[d]_{E \otimes_B \iota_M} && E \otimes_B M' \ar[d]^{E \otimes_B \iota_{M'}} \\
E \otimes_B M_B \ar[rr]_-{E \otimes_B f_B} && E \otimes_B M'_B }$$ in which $\iota_{M}:M \to M_B$ and $\iota_{M'}:M' \to M'_B$ are the
canonical inclusions and using that $\iota_{M}$ (and hence also $E \otimes_B \iota_{M}$) is a split monomorphism, one sees that
$E \otimes_B f$ is a monomorphism. Thus, $E$ is a (faithfully) flat $B$-module. This implies that the class $\mathfrak{P}$ does not satisfies Condition (IV) by exactly the same argument as in Proposition \ref{Pure}.}
\end{remark}

{\bf Acknowledgements.} We are grateful to George Janelidze for valuable discussions
on the subject of this paper.

\bigskip

\noindent
{\bf Address:} \\[+1mm]
{A. Razmadze Mathematical Institute of I. Javakhishvili Tbilisi State University,\\
2, University st., Tbilisi
0186,  } {\small and} \\
 {Tbilisi Centre for Mathematical Sciences,\\
Chavchavadze Ave. 75, 3/35, Tbilisi 0168}, \\
Republic of Georgia.\\
    {\small bachi@rmi.ge}\\[+1mm]

\end{document}